\documentclass[11pt,a4paper]{article}
 \usepackage{euscript}
\usepackage{amsmath}
\usepackage{graphicx}
\usepackage{amsthm}
\usepackage{amssymb}
\usepackage{epsfig}
\usepackage{url}
\usepackage{colonequals}
\usepackage{amsmath,amscd}
\theoremstyle{theorem}
\newtheorem{theorem}{Theorem}[section]
\newtheorem{corollary}[theorem]{Corollary}

\newtheorem{lemma}[theorem]{Lemma}
\newtheorem{proposition}[theorem]{Proposition}

\newtheorem{definition}{Definition}[section]

\newtheorem{example}{Example}[section]
\newtheorem{remark}{Remark}[section]
\numberwithin{equation}{section}
 
\newcommand{\s}{\sigma}

\DeclareSymbolFont{symbolsC}{U}{txsyc}{m}{n}
\DeclareMathSymbol{\notniFromTxfonts}{\mathrel}{symbolsC}{61}
\title{Skew-convex function rings and evaluation of skew rational functions} 
\author{Masood Aryapoor\\
\tiny{\textit{Division of Mathematics and Physics}}\\
\tiny{\textit{M\"{a}lardalen  University}}\\
\tiny{\textit{Hamngatan 15, 632 17, Eskilstuna, 
Sweden
}}
}
 \date{}
\usepackage{microtype} 
\begin{document}
 \maketitle
\begin{abstract}
\noindent
The product formula for evaluating products of skew polynomials is used to construct a class of rings.  As an application, we present a method of evaluating quotients of skew polynomials.      
\end{abstract}
\begin{section}{Introduction} 
We present the notion of the ``skew product" which is a particular binary operation on sets of functions with values in a fixed skew field (see Definition \ref{(D)product}). The concept of the skew product is motivated by the ``product formula" for evaluating skew polynomials introduced by Lam and Leroy \cite[Theorem 2.7]{LamLeory1988}. We show that the skew product gives rise to near-ring structures. Restricting our attention to the class of ``skew-convex" functions, we arrive at the notion of skew-convex function rings (see Definition \ref{(D)skew convex} and Theorem \ref{(T)Thetwistedring}). It turns out that skew-convex function rings are closely related to endomorphism rings (see Section 2.2). The method of evaluating skew polynomials has found interesting applications (see, for example, \cite{MR4586427, LamLeory1988, MR2025912, lam1994hilbert}). Using skew-convex function rings, we extend the method of evaluating skew polynomials to a method of evaluating quotients of skew polynomials.

The paper is organized as follows. Subsection 2.1 introduces the notions of the skew product and skew-convex functions, and gives their basic properties. Subsection 2.2 deals with some general structural results. In Subsection 2.3, we study skew-invertible functions, that is, functions that are invertible with respect to the skew product. 
 Section 3 deals with evaluating quotients of skew polynomials.  
  
\end{section} 


\begin{section}{Skew-convex function rings}
In \cite{LamLeory1988}, the authors presented a method of evaluation of skew polynomials over a skew field using which skew polynomials can naturally be considered as functions on the ground skew field.    
It turns out that the value of the product of two skew polynomials at a given point may not be equal to the product of the values of the skew polynomials at the same point. The correct formula for evaluating products of skew polynomials, called the product formula, is given in  \cite[Theorem 2.7]{LamLeory1988}. The product formula can be regarded as a binary operation on functions which we shall call the \textit{skew product} (see Formula \ref{(D)product}). This section introduces the notion of the skew product, and gives  some general facts regarding the skew product. 

\begin{subsection}{Skew-convex function rings}
	Let $K$ be  a skew field and $X$ be a (nonempty) set on which the multiplicative group $K^*\colonequals K\setminus \{0\}$ acts (on the left). The action of $a\in K^*$ on $x\in X$ is denoted by $ ^{a}x$. We will freely use the standard terminology of Group Theory. In particular, we use the following notions: An \textit{invariant} subset of $X$ is a set $Y\subset X$ such that $^{a}y\in Y$ for all $a\in K^*$ and $y\in Y$; An \textit{orbit} is a nonempty invariant subset of $X$ which is minimal with respect to inclusion. 
	We denote the set of all functions $f\colon  X\to K$ by $\mathcal{F}(X)$. By abuse of notation, a constant function in $ \mathcal{F}(X)$, whose value is $a\in K$,  is simply denoted by $a$.
	Given functions $f,g\colon  X\to K$,  we let $f+g$ denote the pointwise addition of the functions $f$ and $g$. 
	\begin{definition}
		The \textit{left skew product} of functions $f,g\in \mathcal{F}(X)$ is a function $f\diamond g:X\to K$ defined as follows
		\begin{equation}\label{(D)product}
			(f\diamond g)(x)=\begin{cases}
				f\left(  \, ^{{g(x)}}x\, \right) g(x) & \text{if } g(x)\neq 0,\\
				0& \text{if } g(x)=0.
			\end{cases}
		\end{equation}
	\end{definition}
	The \textit{right skew product} is defined using the formula
	\begin{equation}\label{(D)rightproduct}
		(f\diamond_r g)(x)=\begin{cases}
			f(x)g\left(  \, ^{{f(x)^{-1}}}x\, \right)  & \text{if } f(x)\neq 0,\\
			0& \text{if } f(x)=0.
		\end{cases}
	\end{equation}

	In this paper, we will exclusively work with the left skew product.  Therefore, we shall drop the adjective ``left".  We leave it to the reader to formulate and prove  similar  results for the right skew product.
	 
	It is easy to see that $(a\diamond f)(x)=af(x)$, for every $a\in K,x\in X$. We shall henceforth denote $a\diamond f$ by $af$. Note that $(a,f)\mapsto af$ turns $ \mathcal{F}(X)$ into a left $K$-vector space. 
	In the following lemma, the proof of which is straightforward, we collect some properties of the skew product. 
	\begin{lemma}\label{(L)propertiesofsproduct}
		Let $f,h,g\colon  X\to K$ be arbitrary functions. Then:\\
		(1) The constant function $1$ is a unit for $\diamond$, that is, $f=f\diamond 1=1\diamond f$.\\
		(2) $(f+g)\diamond h=  f\diamond h+ g\diamond h$, that is, the right distributive law (with respect to pointwise addition) holds for the skew product.  \\
		(3) $ (f\diamond g) \diamond h = f \diamond (g\diamond h)$, that is, $\diamond$ is associative. 
	\end{lemma}
	It follows from this lemma that the set $\mathcal{F}(X)$ equipped with pointwise addition and the skew product is a structure known as ``right near-ring" in the literature (see \cite{MR721171}). 
	We note  that the skew product may not be left distributive with respect to pointwise addition. However, the left distributive law  holds for a class of functions  described below. 
	\begin{definition}\label{(D)skew convex}
		A function $f\colon  X\to K$ is called skew convex if 
		$$f\diamond(a+b)=f\diamond a+f\diamond b, \text{ for all } a,b\in K.$$
	\end{definition}	
	The set of all skew-convex functions $f\colon  X\to K$ is denoted by  $\mathcal{S}(X)$. Any constant function belongs to $\mathcal{S}(X)$ since $a\diamond b=ab$ for all $a,b\in K$. In particular, $K$ is a subring of $\mathcal{S}(X)$. More generally, we have the following result. The  easy proof is left to the reader.
	\begin{proposition}\label{(P)constantskewfunction}
		A function $f\colon  X\to K$ which is constant on every orbit in $X$, is skew convex. 
	\end{proposition}
	The following lemma gives an important property of skew-convex functions.
	\begin{lemma}\label{(L)leftdist}
		Let $h\colon  X\to K$ be given. The condition 	$$ h\diamond(f+g) = h\diamond f + h\diamond g,$$
		holds for all functions $f,g\colon  X\to K$ if and only if $h$ is skew convex.
	\end{lemma}
	\begin{proof}
	The result follows from the identity 
	$$(h\diamond f)(x)=(h\diamond {f(x)})(x), \text{ for all } f\in \mathcal{F}(X) \text{ and } x\in X.$$
	\end{proof}	
	As a consequence of this lemma, we have the following result.  
	\begin{theorem}\label{(T)Thetwistedring}
	Equipped with the left skew product, the additive group $\mathcal{S}(X)$  is a ring with identity. 
	\end{theorem}
	\begin{proof}
		The result follows from Lemma \ref{(L)leftdist} and the  general fact that in any right near-ring $R$, the set $$\{r\in R\,|\, r(s_1+s_2)=rs_1+rs_2\text{ for all } s_1,s_2\in R\},$$
		is a ring. 
	\end{proof}
	We call $\mathcal{S}(X)$, equipped with pointwise addition and the skew product,  \textit{the ring of skew-convex functions} on $X$ determined by the action of $K^*$ on $X$. We now give some examples of skew-convex function rings. 
	\begin{example}
		If the action of $K^*$ on $X$ is trivial, the ring $\mathcal{S}(X)$ is just the familiar ring of all functions $f\colon  X\to K$ equipped with pointwise addition and pointwise multiplication. 
	\end{example}
	The following example justifies the terminology and explains the link between skew-convex function rings and skew polynomial rings. For an introduction to skew polynomial rings, we refer the reader to \cite{Freeidealrings}. 
	\begin{example}\label{(Ex)(s,delta)-skewmaps}
		Let $\s\colon  K\to K$ be an endomorphism and $\delta\colon  K\to K$ be a $\s$-derivation. Let $K[T;\s,\delta]$ denote the ring of skew polynomials determined by $\s$  and $\delta$. Every nonzero element of $K[T;\s,\delta]$ can uniquely be written as $\sum_{m=0}^na_mT^m$ where $a_m\in K$ with $a_n\neq 0$. The identity $Ta=\s(a)T+\delta(a)$, where $a\in K$, holds in $K[T;\s,\delta]$. Following \cite{LamLeory1988}, we consider  the $(\s,\delta)$-action of $K^*$ on $K$, that is, 
		\begin{equation}\label{(E)conjugate}
			^{b}a=\s(b)ab^{-1}+\delta(b)b^{-1}.
		\end{equation} 
		The ring of skew-convex functions determined by the $(\s,\delta)$-action  is denoted by $K[\s,\delta]$. One can verify that there exists a unique ring homomorphism $$K[T;\s,\delta]\to K[\s,\delta],$$ which sends each $a\in K$ to the constant function $a$, and $T$ to the identity function $id\colon  K\to K$. In particular,  every skew polynomial $P(T)\in K[T;\s,\delta]$ can, under this homomorphism, be considered as a skew-convex function $P\colon  K\to K$. The reader can  verify that the value $P(a)$ of $P$  at $a\in K$ coincides with the evaluation map introduced in \cite{LamLeory1988}, that is, $P(a)$ is the unique element of $K$ for which we have
		$$P(T)-P(a)\in K[T;\s,\delta](T-a).$$ 
		Let us remark that the product formula for evaluating products of skew polynomials is an important consequence of the existence of the above ring homomorphism.   
	\end{example} 
\begin{example}\label{(Ex)regularaction}
	Consider the left regular action of $K^*$ on $K^*$, i.e., ${}^{b}a=ba$. The reader can easily verify that a function $f:K^*\to K$ is skew-convex (with respect to the left regular action) iff there exists a group homomorphism $\phi_f:(K,+)\to (K,+)$ such that $f(x)=\phi_f(x)x^{-1}$ for all $x\in K^*$.   It is straightforward to check that the assignment $f\mapsto \phi_f$ establishes an isomorphism between the ring of skew-convex functions on $K^*$ and the endomorphism ring $End(K,+)$. For a more general result, see Proposition \ref{(P)transitive}. 
\end{example}
	
\end{subsection}
\begin{subsection}{Structure of skew-convex function rings}
We begin with a result regarding homomorphisms between skew-convex function rings for which we need some preliminaries.  
Let $K$ be a skew field, and $X,Y$ be (nonempty) sets on which $K^*$ acts (on the left). For a map $\phi\colon  X\to Y$, let  $\phi^*\colon  \mathcal{F}(Y)\to \mathcal{F}(X)$ denote the pullback map  $\phi^*(f)=f\circ \phi$.  Recall that a map $\phi\colon  X\to Y$ is called \textit{action-preserving} if  $\phi(^{a}x)={}^{a}\phi(x)$ for all $a\in K^*$ and $x\in X$.
\begin{proposition}\label{(P)actionpreserving}
	Assume that $\phi\colon  X\to Y$ is action-preserving. Then, for any function $f\in \mathcal{S}(Y)$, we have $\phi^*(f)\in \mathcal{S}(X)$. Moreover, the map 
	$\phi^*\colon  \mathcal{S}(Y)\to \mathcal{S}(X)$ is a homomorphism of rings. 
\end{proposition}
\begin{proof}
	The fact that $\phi^*(f)\in \mathcal{S}(X)$, for any $f\in \mathcal{S}(Y)$, follows from the identity $$(f\circ \phi)\diamond a=(f\diamond a)\circ \phi, \text{ where } a\in K.$$ 
	The rest of the proof is straightforward. 
\end{proof}
The following result reduces the problem of classifying skew-convex function rings to the case of transitive actions, i.e., actions with a single orbit. 
\begin{proposition}\label{(P)structure}
	Let $X_i, i\in I,$ be the family of all orbits in $X$. Then, the ring $\mathcal{S}(X)$ is isomorphic to  the direct product of the rings $\mathcal{S}(X_i), i\in I$. 
\end{proposition}
\begin{proof}
	For any function $f:X\to K$, let $f_i$ denote the restriction of $f$  to the set $X_i$. It is easy to check that the assignment $f\mapsto (f_i)_{i\in I}$ gives an isomorphism between  $\mathcal{S}(X)$ and the direct product $\prod_{i\in I} \mathcal{S}(X_i)$. 
\end{proof}	
Any transitive action of $K^*$ is isomorphic to a left regular action in the sense that it is of  the form 
$${}^{a}(bG)=(ab)G,$$
where $G$ is a subgroup of $K^*$ and $K^*/G=\{bG|\,b\in K^*\}$ is the set of all left cosets of $G$ in $K^*$. A map $\phi:K\to K$ is called \textit{right $G$-linear} if $\phi(a-b)=\phi(a)-\phi(b)$ and $\phi(ac)=\phi(a)c$, for all $a,b\in K$ and $c\in G$.  The following proposition gives a characterization of skew-convex function rings for the case of transitive actions. 
\begin{proposition}\label{(P)transitive}
	Let $G$ be a subgroup of $K^*$ and consider the left regular action of $K^*$ on  $K^*/G$. Then:\\
	(1) A function $f:K^*/G\to K$ is skew convex iff there exists a (unique) right $G$-linear map $\phi_f:K\to K$ such that $f(xG)=\phi_f(x)x^{-1}$ for all $x\in K^*$. \\
	(2) The assignment $f\mapsto \phi_f$ establishes an isomorphism between $\mathcal{S}(K^*/G)$ and the  endomorphism ring $End(K_G)$ of right $G$-linear operators on $K$. 
\end{proposition}
\begin{proof}
	Given a function $f:K^*/G\to K$, we define $\phi_f:K\to K$ as follows
	$$\phi_f(x)=
	\begin{cases}
		f(xG)x& \text{if } x\neq 0,\\
		0 & \text{if } x=0.
	\end{cases}
	$$
	It is straightforward to check that $f$ is skew-convex iff $\phi_f$ is right $G$-linear. The easy proof of (2) is left to the reader. 
\end{proof}   	
We end this section with a remark.
\begin{remark}\label{(Re)isomorphism}
	Keeping the notations as in Example \ref{(Ex)(s,delta)-skewmaps}, let $a\in K$ be fixed. The $(\s,\delta)$-conjugacy class of $a$, that is, the set 
	$$\Delta^{\s,\delta}(a)\colonequals \{\,{}^{b}a\, |\,b\in K^*\},$$
	is an invariant subset of $K$, and  the $(\s,\delta)$-action   is transitive on $\Delta^{\s,\delta}(a)$. One can show that the set 
	$$C^{\s,\delta}(a)\colonequals \{b\in K^*|\,{}^{b}a=a\}\cup\{0\},$$
	is a skew subfield of $K$ (see Lemma 3.2 in \cite{LamLeory1988}). An application of Proposition \ref{(P)transitive} reveals that the  ring of skew-convex functions on $\Delta^{\s,\delta}(a)$ is isomorphic to the ring of right $C^{\s,\delta}(a)$-linear operators on $K$. This isomorphism has implicitly been used in the proof of  \cite[Proposition 3.16]{LamLeory1988}. In particular, we obtain a ring homomorphism 
	$$\lambda:K[T;\s,\delta]\to End(K_{C^{\s,\delta}(a)}).$$
	This ring homomorphism coincides with the homomorphism $\Lambda_a$ introduced and studies in \cite{leroy2012noncommutative} (see Corollary 1.13 in the reference). We also note that $\lambda(P(T))=\lambda_{P,a}$ where $\lambda_{P,a}$ is the so-called $\lambda$-transform defined in Definition 4.10 of \cite{MR2025912}.
\end{remark}	
\end{subsection}
\begin{subsection}{Skew-invertible  functions}

	As before, let $K$ be a skew field and $X$ be a set on which $K^*$ acts. A function $f\in \mathcal{F}(X)$ is called \textit{skew invertible} if it is invertible with respect to the skew product, in which case, the inverse of $f$ with respect to the skew product is called its \textit{skew inverse} and denoted by $f^{\langle-1\rangle}$. 
	\begin{lemma}\label{(L)Inverseofc-invariant}
		 If $f\in \mathcal{S}(X)$ is skew invertible, then the skew inverse of $f$ belongs to $\mathcal{S}(X)$.   
	\end{lemma}
	\begin{proof}
		By Lemma \ref{(L)leftdist}, we have 
		$$\mathcal{S}(X)=\{f\in \mathcal{F}(X)\,|\, f\diamond(g+h)=f\diamond g+f\diamond h\text{ for all } g,h\in \mathcal{F}(X)\}.$$
		The result follows from this identification and the following general fact whose proof is left to the reader: Let $R$ be a right near-ring and consider the ring 
		$$R'=\{r\in R\,|\, r(s_1+s_2)=rs_1+rs_2\text{ for all } s_1,s_2\in R\}.$$
		If $r\in R'$ is invertible in $R$, then its inverse belongs to $R'$. 
	\end{proof}
	   Next, we give a characterization of skew-invertible functions.
	\begin{lemma}\label{(L)Invertibility} 
		(1)  Let $f\in \mathcal{F}(X)$. There exists $g\in \mathcal{F}(X)$ such that $f\diamond g=1$ if and only if for every $x\in X$, there exists some $a\in K^*$ such that $f({}^ax)=a^{-1}$.\\
		(2) Let $g\in \mathcal{F}(X)$.  There exists $f\in \mathcal{F}(X)$ such that $f\diamond g=1$ if and only if $g(X)\subset K^*$, and the map $x\mapsto {}^{g(x)}x$ is 1-1. \\
		(3) A function $f\in \mathcal{F}(X)$ is skew invertible if and only if $f(X)\subset K^*$ and the assignment $x\mapsto {}^{f(x)}x$ establishes a bijection from $X$ onto $X$. 
	\end{lemma}
	\begin{proof}
		(1)  The proof is straightforward.\\
		(2)   The proof of the ``if" direction may be left to the reader. To prove the other direction, let $g$ satisfy the stated properties. We define a map $f\colon  X\to K$ as follows: If $y={}^{g(x)}x$ for some $x\in X$, we set $f(y)=g(x)^{-1}$. Otherwise, we set $f(y)=x_0$ where $x_0\in X$ is a fixed element. It is easy to see that $f$ is well-defined and  $f\diamond g=1$. \\
		(3) Assume that $f$ is skew invertible and let $g$ be its skew inverse. By $(2)$,  $f$ is nonzero on $X$ and $x\mapsto {}^{f(x)}x$ is 1-1. The fact that $x\mapsto {}^{f(x)}x$ is onto follows from the following identity
		$${}^{f(^{g(y)}y)}\left( {}^{g(y)}y\right)=y.$$ 
		Conversely, assume that $f$ satisfies the stated properties. By Part $(2)$, there exists $g\in \mathcal{F}(X)$ such that $g\diamond f=1$. We need only  show that $f\diamond g=1$. Given an arbitrary element $x\in X$, we can choose $y\in X$ such that $x={}^{f(y)}y$. We have
		$$(g\diamond f)(y)=1\implies g({}^{f(y)}y)f(y)=1\implies g(x)f(y)=1.$$
		Therefore, we have ${}^{g(x)}x={}^{g(x)f(y)}y=y$ from which it follows that 
		$$g(x)f(y)=1\implies f(y)g(x)=1\implies f({}^{g(x)}x)g(x)=1 \implies (f\diamond g)(x)=1.$$
		Since $x\in X$ was arbitrary, we conclude that $f\diamond g=1$.
	\end{proof}
\begin{remark}
	The map $x\mapsto {}^{f(x)}x$ has also been used in the context of W-polynomials (see the definition of the $\Phi$-transform in  \cite[Definition 4.5]{MR2025912}). 
\end{remark}
	Regarding skew-invertible elements of $\mathcal{S}(X)$, we have the following characterization. 
\begin{proposition}\label{(P)inverserightinverse}
	A skew-convex function $f:X\to K$  is skew invertible if and only if $f(X)\subset K^*$ and for any $x\in X$, there exists some $a\in K^*$ such that $f({}^ax)=a^{-1}$.  
\end{proposition}

	\begin{proof}
		The ``only if" direction follows from Part 1 of Lemma \ref{(L)Invertibility}. To prove the other direction, we use  Part (2) of Lemma \ref{(L)Invertibility}. Therefore, we need only show that if ${}^{f(x)}x={}^{f(y)}y$, then $x=y$. Let ${}^{f(x)}x={}^{f(y)}y$ for some $x,y\in X$. Then $y$ and $x$ are in the same orbit, implying that there exists $a\in K^*$ such that $y={}^{a}x$. So, 
		$${}^{f(x)}x={}^{f(y)}y={}^{f({}^{a}x)a}x.$$
		It follows that $f(x)b=f({}^ax)a$ for some $b\in K^*$ satisfying ${}^bx=x$. Note that $f({}^{-a}x)=f({}^ax)$ because 
		$$0=f\diamond (a+(-a))=f\diamond a +f\diamond (-a)\implies f\diamond (-a)=-f\diamond a.$$ 
		Therefore, we can write
		$$0=f({}^bx)b-f({}^{-a}x)a=(f\diamond b+f\diamond (-a))(x)=\left( f\diamond (b-a)\right)(x).$$
		It follows that $b=a$, since  $f(X)\subset K^*$. Thus  $x={}^bx={}^ax=y$.
	\end{proof}

\end{subsection}


\end{section} 

\begin{section}{Evaluation of skew rational functions}
	  The method of evaluating skew polynomials introduced in \cite{LamLeory1988} is a natural generalization of evaluation of polynomials in the commutative setting. Therefore, it is of importance to extend the method to  skew rational functions, that is, quotients of skew polynomials. In this section, using the material developed in the previous section, we present a method of evaluating skew rational functions. Let us fix some notations. Throughout this section, $K$ is a skew field,  $\s\colon  K\to K$ is an endomorphism and $\delta\colon  K\to K$ is a $\s$-derivation. We will work with the $(\s,\delta)$-action (see Example \ref{(Ex)(s,delta)-skewmaps}). The $(\s,\delta)$-conjugacy class of an element $a\in K$ is denoted by $\Delta^{\s,\delta}(a)$.  If $a,b\in K$ are $(\s,\delta)$-conjugate, we write $a\sim b$. The notation $a\nsim b$  means  $b\notin \Delta^{\s,\delta}(a)$.
	\begin{subsection}{Basic definitions}
		We begin with some general facts. For more details, see \cite{Freeidealrings}. 	It is known that $K[T;\s,\delta]$ is a left PID, and therefore, a left Ore domain. Its field of fractions (called the skew rational function field) is denoted by $K(T;\s,\delta)$. Every skew rational function can be represented as a left quotient $P(T)^{-1}Q(T)$ for some skew polynomials  $0\neq P(T), Q(T)$, where  $P(T)$ is called the denominator and $Q(T)$ is called the numerator of the quotient. Although such a representation is not unique for a given $f(T)\in K(T;\s,\delta)$, there exists a unique representation  $P(T)^{-1}Q(T)$ of $f(T)$ such that $P(T)$ is monic and has the least possible degree among all representations of $f(T)$. This representation will be called the \textit{minimal representation} of $f(T)$. 
		
		After the above preliminaries, we shall describe a method of evaluating skew rational functions. Since in the commutative setting, a rational function is not defined at the roots of its denominator, some care is needed in evaluating skew rational functions. Therefore, we introduce the following definition. 
		\begin{definition}
			A skew rational function $f(T)\in K(T;\s,\delta)$ is said to be defined at $a\in K$ if 
			the denominator of the minimal representation of $f(T)$ is skew-invertible as a function on the $(\s,\delta)$-conjugacy class $\Delta^{\s,\delta}(a)$ of $a$. 		 
		\end{definition}
		We now define the evaluation of a skew rational function at elements of $K$. Recall that every skew polynomial $P(T)$ can be regarded as a skew-convex function on any invariant set $A\subset K$ (see Example \ref{(Ex)(s,delta)-skewmaps}). By abuse of notation, the skew-convex function associated to $P(T)$ is denoted by $P:A\to K$. 
		\begin{definition}
			Let $f(T)\in K(T;\s,\delta)$ have the minimal representation $P(T)^{-1}Q(T)$. Assume that $f(T)$ is defined at $a\in K$. The value of $f(T)$ at $a$ (denoted by $f(a)$) is defined to be 
			$$f(a)\colonequals \left(P^{\langle-1\rangle}\diamond Q \right)(a),$$
			 where  $P^{\langle-1\rangle}:\Delta^{\s,\delta}(a)\to K$ is the skew-inverse of $P:\Delta^{\s,\delta}(a)\to K$. 		 
		\end{definition}
		It is convenient to introduce one more definition. 
		\begin{definition}
			Let $f(T)\in K(T;\s,\delta)$. The set of all $a\in K$, at which $f(T)$ is defined, is denoted by $dom(f)$, and called the domain of $f(T)$. The function sending $a\in dom(f)$ to $f(a)$ is denoted by $f:dom(f)\to K$.  		 
		\end{definition}
		  It follows from the definitions that if $f(T)\in K(T;\s,\delta)$ is defined at some $a\in K$, then it is also defined at all $c\in \Delta^{\s,\delta}(a)$. Therefore, $dom(f)$ is a $(\s,\delta)$-invariant subset of $K$, for all $f(T)\in K(T;\s,\delta)$. 
		 
		 In general, it is difficult to find the domain of an arbitrary skew rational function. Here, we present a partial result. 
		 Recall that a skew polynomial $P(T)\in K[T;\s,\delta]$ is called semi-invariant if $P(T)K\subset KP(T)$. For a detailed account of semi-invariant skew polynomials, we refer the reader to \cite{lamleroy1988algebraic}.
		 \begin{proposition}\label{(P)semi-invariant}
		 	Assume that $\s:K\to K$ is an automorphism. Let $f(T)\in K(T;\s,\delta)$ have the minimal representation $P(T)^{-1}Q(T)$ such that $P(T)$ is semi-invariant and of degree $n$. Then, for any $a\in K$, $f(T)$ is defined at $a$ iff $P(a)\neq 0$. Furthermore, we have
		 	$$
		 	f(a)=\s^{-n}\left( P({}^{Q(a)}a)\right)^{-1}Q(a) ,\, \text{ for all }\, a\in K \, \text{ such that }\, P(a)\neq 0. 
		 	$$ 
		 \end{proposition}  
	 \begin{proof}
	 	Let $n$ be the degree of $P(T)$. Fix $a\in K$. For every $x\in K$, we have $P(T)x=\s^n(x)P(T)$ (see \cite[Lemma 2.2]{lamleroy1988algebraic}). Evaluating at $a$, we obtain $P({}^{x}a)x=\s^n(x)P(a)$,  using which one can prove the proposition. The details are left to the reader.  
	 \end{proof} 
 Let us give an example illustrating the proposition. 
 \begin{example}
 	Let $\s$ be an involution and $\delta=0$. Let $b\in K$ belong to the center of $K$. Then, the polynomial $P(T)=T^2+b$ is semi-invariant (see also Example 2.5.(a) in \cite{lamleroy1988algebraic}). By Proposition \ref{(P)semi-invariant}, $f(T)=P(T)^{-1}$ is defined at $a\in K$ iff $\s(a)a+b\neq 0$. Moreover, we have  
 	$$f(a)=(\s(a)a+b)^{-1}, \,\text{ for all } \, a\in K \text{ with }\, \s(a)a+b\neq 0.$$
 \end{example}
		 We end this part with some remarks. 
		\begin{remark}
			Our method of evaluation of skew rational functions has some features not present in the commutative setting. One such feature is that it may happen that $dom(f)$ is empty for every skew rational function not in $K[T;\s,\delta]$. 
			There exist examples of $K[T;\s,\delta]$ in which every irreducible skew polynomial is of degree one, and  $\Delta^{\s,\delta}(a)=K$ for all  $a\in K$. Examples are universal differential fields discovered by Kolchin (\cite{MR58591}). For such examples, we have $dom(f)=\emptyset$, for every skew rational function $f$ not in $K[T;\s,\delta]$.  
		\end{remark}
	The following remarks deal with some equivalent formulations of the above definitions. 
		\begin{remark}
		Working with the notations of Corollary 1.13 in \cite{leroy2012noncommutative}, we can see that $f(T)=P(T)^{-1}Q(T)$ is defined at $a\in K$ iff $P(T_a):V\to V$ is a bijection. Moreover, the value of $f(T)$ at $a\in K$ is 
		$$f(a)=\left( P(T_a)^{-1}\circ Q(T_a)\right)(1)=P(T_a)^{-1}(Q(a)).$$
		This approach has the advantage that one can define evaluation of skew rational functions over arbitrary rings. 
	\end{remark}
\begin{remark}
	Using the  $\lambda$-transform (see Definition 4.10 in \cite{MR2025912}), we can see that $f(T)=P(T)^{-1}Q(T)$ is defined at $a\in K$ iff $\lambda_{P,a}:K\to K$ is a bijection.
\end{remark}
	\end{subsection}
	\begin{subsection}{Evaluation of the skew rational function $(T-b)^{-1}$}
 	This part deals with evaluation of skew rational functions of the form \newline $(T-b)^{-1}$, where $b\in K$. 
			\begin{proposition}\label{(P)T-b}
				Let $a,b\in K$. Then,\\
				(1) $(T-b)^{-1}$ is defined at $a$ iff  $b\nsim a$ and the $(\s,\delta)$-metro equation 
				$$\s(x)c+\delta(x)-bx=1,$$
				has a solution $x\in K$, for all $c\in \Delta^{\s,\delta}(a)$. \\
				(2) If $(T-b)^{-1}$ is defined at $a$, then the value  of $(T-b)^{-1}$ at $a$ is the (unique) solution $x$ of  the $(\s,\delta)$-metro equation 
				$$ \s(x)a+\delta(x)-bx=1.$$
			\end{proposition}
		\begin{proof}
			(1) Using Lemma \ref{(L)Invertibility} and Proposition \ref{(P)inverserightinverse}, we see that $T-b$ is skew invertible as an element of $\mathcal{S}(\Delta^{\s,\delta}(a))$ if and only if (1) $T-b$ does not have a root in $\Delta^{\s,\delta}(a)$, and (2) for any $c\in \Delta^{\s,\delta}(a)$, there exists $x\in K^*$ such that $(T-b)({}^xc)=x^{-1}$. The first condition is equivalent to $b\notin \Delta^{\s,\delta}(a)$. The second condition is equivalent to saying that for any $c\in A$, there exists $x\in K^*$ satisfying the equation 
			${}^xc-b=x^{-1}$, or equivalently, the $(\s,\delta)$-metro equation  $$\s(x)c+\delta(x)-bx=1.$$
			(2) This is a direct consequence of (1) and the definition of the skew product. 
		\end{proof}
	In \cite{MR2025912}, the notion of a $(\s,\delta)$-metro equation is  studied in the context of Wedderburn polynomials (also called W-polynomials).  Using the results of \cite{MR2025912}, we obtain the following criterion.
			\begin{corollary}
					Let $a,b\in K$. Then, $(T-b)^{-1}$ is defined at $a$  iff $b\nsim a$ and $(T-c)(T-b)$ is a W-polynomial for every $c\in \Delta^{\s,\delta}(a)$. In particular, if $\Delta^{\s,\delta}(a)$ is an algebraic $(\s,\delta)$-conjugacy class, then $(T-b)^{-1}$ is defined at $a$, for all  $b\nsim a$.
			\end{corollary}
			\begin{proof}
				The first statement follows from Proposition \ref{(P)T-b} and  \cite[Theorem 6.6]{MR2025912}. The second statement follows from \cite[Corollary 6.7]{MR2025912}. 
			\end{proof}
		The following result sheds light on evaluating skew rational polynomials of the type $(T-b)^{-1}$. 
		\begin{proposition}\label{(P)conjugateT-b}
			Let $a,b,c,d\in K$ satisfy $a\sim c$ and $b\sim d$. If $(T-b)^{-1}$ is defined at $a$, then $(T-d)^{-1}$ is defined at $c$. 
		\end{proposition}
	\begin{proof}
		The result follows from the definition and the identity 
		$$T-d=\s(x)^{-1}(T-b)x,$$ where $x\in K^*$ satisfies $b={}^{{x}}d$. \\
	\end{proof}

	We now give some examples. 
		\begin{example}
		Let $K[T;\s]$ be a skew polynomial ring of endomorphism type, that is, $\delta=0$. It follows from Proposition \ref{(P)T-b} that the skew rational function $T^{-1}$ is defined at $a\in K$ if and only if $a\neq 0$ and $\Delta^{\s,\delta}(a)\subset \s(K)$. Furthermore, the value of $T^{-1}$ at such an element $a$ is 
		$\s^{-1}\left(a^{-1}\right)$. In particular, if $\s$ is an automorphism, then $dom(T^{-1})=K^*$. 
	\end{example} 
		\begin{example}\label{(Ex)Complex}
			Consider the skew polynomial ring $\mathbb{C}[T;\bar{}\bar{}\,]$ where $\bar{}\bar{}\,$ is the complex conjugate map and $\delta=0$. Using Proposition  \ref{(P)T-b}, one can show that  the domain of $(T-b)^{-1}$, where $b\in \mathbb{C}$, is the set
			$\{z\in\mathbb{C}|\, |z|\neq|b|\}.$
			Furthermore, the value of $f(T)=(T-b)^{-1}$ at $z\in dom((T-b)^{-1})$ is 
			$$ f(z)=\frac{z+\overline{b}}{|z|^2-|b|^2}.$$
			We remark that similar results hold more generally for the case when\newline $\s:K\to K$ is an involution of a commutative field $K$.  
		\end{example}
	 
		\begin{example}\label{(Ex)quaternions}
			Consider the ring $\mathbb{H}[T]$ of polynomials over the skew field of quaternions in a central indeterminate $T$.  Using Proposition \ref{(P)T-b}, one can show that  the domain of $(T-q_0)^{-1}$, where $q_0\in \mathbb{H}$, is the set of all $q\in \mathbb{H}$ not conjugate to $q_0$. 
			Moreover, the value of $f(T)=(T-q_0)^{-1}$ at $q\in dom((T-q_0)^{-1})$ is 
			$$ f(q)=(q-\overline{q}_0)\left( q^2-2\text{Re}(q_0)q+|q_0|^2\right)^{-1}.$$
			This function plays a central role in  the version of ``quaternionic analysis"  introduced in \cite{aryapoor2022skew}. It has also been used as a Cauchy kernel, see \cite{ACauchykernel}. 
		\end{example} 
	
	\end{subsection}
	\begin{subsection}{Evaluation of skew rational functions over centrally finite skew fields}
	For $a\in K$, we set 
	$$C^{\s,\delta}(a)\colonequals \{b\in K^*\,|\, {}^{b}a=a\}\cup\{0\}.$$
	Note that $C^{\s,\delta}(a)$ is a skew subfield of $K$ (see \cite[Lemma 3.1]{LamLeory1988}). 
	\begin{proposition}\label{(P)finite}
		Let $a\in K$ such that $K$ is finite-dimensional as a right $C^{\s,\delta}(a)$-vector space. Then, a skew rational function with minimal representation $P(T)^{-1}Q(T)$ is defined at $a$ iff $P(c)\neq 0$ for all $c\in \Delta^{\s,\delta}(a)$. 
	\end{proposition}
	\begin{proof}
		The ``only if" direction is trivial. To prove the other direction, one can use Proposition \ref{(P)transitive} (see also Remark \ref{(Re)isomorphism}) and the fact that  $K$ is finite-dimensional as a right $C^{\s,\delta}(a)$-vector space. Alternatively, one can use Part (a) of Theorem 6.2 in \cite{MR2025912}.  
	\end{proof}
	\begin{remark}\label{(Re)conjugacydimension}
		It is known that $K$ is finite-dimensional as a right $C^{\s,\delta}(a)$-vector space iff the conjugacy class of $a$ is $(\s,\delta)$-algebraic (see Proposition 4.2 in \cite{lamleroy1988algebraic}). 
	\end{remark}
	As special cases of this proposition, we give the following results. 
		\begin{corollary}
		Assume that every conjugacy class of $K$ is $(\s,\delta)$-algebraic. Then, for any irreducible element $P(T)\in K[T;\sigma,\delta]$ of degree $>1$, the skew rational function $P(T)^{-1}$ is defined everywhere, i.e., $dom(P(T)^{-1})=K$. 
	\end{corollary}
	\begin{corollary}
		Let $K$ be a centrally finite skew field. Assume that $\s$ is the identity homomorphism and $\delta=0$. Then, a skew rational function in $K[T]=K[T;id,0]$ is defined at $a\in K$ iff  the denominator of its minimal  representation does not have a root in the conjugacy class
		$\{bab^{-1}\,|\, b\in K^*\}$. 
	\end{corollary}
Let us now give some examples. 
	\begin{example}
		The skew field $\mathbb{H}$ of quaternions is centrally finite.  A classical result of Niven states that $\mathbb{H}$ is left algebraically closed.  Therefore, every skew rational function $f(T)$ over  $\mathbb{H}$ has a minimal representation of the form
		$$f(T)=\left((T-q_1)(T-q_2)\cdots (T-q_n)\right)^{-1}Q(T).$$
		The domain of $f(T)$ consists of all $q\in \mathbb{H}$ which are not conjugate to any of the $q_i$'s. Note that every conjugacy class in  $\mathbb{H}$ is either a singleton or a 2-dimensional sphere.
	\end{example} 
More generally, we have the following example. 
		\begin{example}
		Assume that every $(\s,\delta)$-conjugacy class in $K$ is algebraic. In the light of Remark \ref{(Re)conjugacydimension}, we see that for every skew rational function $f$,  the complement of $dom(f)$ in $K$ is a finite union of conjugacy classes. 
	\end{example} 
In the light of Proposition \ref{(P)finite}, the following example completes the discussion of evaluating skew rational functions in $\mathbb{C}(T;\bar{}\bar{}\,)$. 
	\begin{example}
			Consider the skew polynomial ring $\mathbb{C}[T;\bar{}\bar{}\,]$ where $\bar{}\bar{}\,$ is the complex conjugate map and $\delta=0$. It is known that every irreducible monic element of  $\mathbb{C}[T;\bar{}\bar{}\,]$ is either of the form $T-a$ or the form $T^2+bT+c$, where $|z|^2+bz+c\neq 0$ for all $z\in \mathbb{C}$ (see Example 1.15.4 in \cite{leroy2012noncommutative}). The case  $T-b$ was treated in Example \ref{(Ex)Complex}. For an irreducible element $P(T)=T^2+bT+c$, the skew rational function $P(T)^{-1}$ is defined everywhere, and moreover, the value of $f(T)=P(T)^{-1}$ at $z\in\mathbb{C}$ is equal to
			$$
			f(a)=\frac{|z|^2-bz+\overline{c}}{||z|^2+c|^2-|zb|^2}.
			$$
	\end{example}
\end{subsection}
\begin{subsection}{The product formula for skew rational functions}	
	 For $a\in K$, we let $Def(a)$ denote the set of all
	skew rational functions that are defined at $a$. The assignment $a\mapsto f(a)$ gives rise to the map $$ev_a:Def(a)\to \mathcal{S}(\Delta^{\s,\delta}(a)),$$ where $\mathcal{S}(\Delta^{\s,\delta}(a))$ is the ring of skew-convex functions on the conjugacy class of $a$. This part deals with $Def(a)$ and $ev_a$, and their properties. 
	
	Fix $a\in K$. Using the observation that the union of any family of left Ore subsets of $K[T;\s,\delta]$ is a left Ore set, we see that there exists a unique left Ore subset $S(a)$ of $K[T;\s,\delta]$ which is maximal with respect to inclusion among all  left Ore subsets $S$ of $K[T;\s,\delta]$ satisfying the property
	$$P(T)\in S\implies  Q(T)^{-1}\in Def(a), \forall\text{ monic irreducible factor $Q(T)$ of } P(T).$$
	Here, the word ``factor" means that $P(T)=Q_1(T)Q(T)Q_2(T)$ for some skew polynomials $Q_1(T),Q_2(T)$. 
	It is clear that the Ore localization $S(a)^{-1}K[T;\s,\delta]$ is a subset of $Def(a)$. We denote $S(a)^{-1}K[T;\s,\delta]$ by $Def_o(a)$.
		\begin{proposition}
		The evaluation map $ev_a:Def_o(a)\to \mathcal{S}(\Delta^{\s,\delta}(a))$ is   a ring homomorphism.
	\end{proposition}
	\begin{proof}
		This follows from the above discussion and the universal property of the Ore localization. 
	\end{proof}
	As an application of this proposition, we present the  product formula for skew rational functions.
	\begin{corollary}
		Let $f(T),G(T)\in  Def_o(a)$ and set $h(T)=f(T)g(T)$. Then
		$$	h(a)=\begin{cases}
			f\left(  \, {}^{{g(a)}}a\, \right) g(a) & \text{if } g(a)\neq 0,\\
			0& \text{if } g(a)=0.
		\end{cases}$$
	\end{corollary}
	In general, the problem of determining $S(a)$ seems to be difficult. Here, we present two partial results. The first result concerns skew polynomials of degree 1. 
	\begin{proposition}
		A polynomial $T-b$ belongs to $S(a)$ iff $(T-b)^{-1}\in Def(a)$. 
	\end{proposition}
	\begin{proof}
		The ``only if" direction is trivial. To prove the other direction, let $(T-b)^{-1}\in Def(a)$.  Consider the set $S$ of all skew polynomials $T-c$ where $c$ is a conjugate of $b$. It is enough to show that $S$ is a left ore set since, by Proposition \ref{(P)conjugateT-b},  $(T-c)^{-1}\in Def(a)$, for all $T-c\in S$. The fact that $S$ is a left Ore set is a consequence of the fact that
		$$(T-{}^{{P(c)}}c)P(T)\in K[T;\s,\delta](T-c),$$
	for any $P(T)\in K[T;\s,\delta]$ with $P(c)\neq 0$. 
	\end{proof}
	Our second result solves the problem in the case of algebraic conjugacy classes. 
	\begin{theorem}
		Assume that $\Delta^{\s,\delta}(a)$ is $(\s,\delta)$-algebraic. Then, a skew polynomial $P(T)$ belongs to $S(a)$ iff $P(T)$ has no (right) roots in the conjugacy class of $a$. In particular, we have $Def_o(f)=Def(a)$.
	\end{theorem}
	\begin{proof}
		The ``only if" direction is proved in Proposition \ref{(P)finite} (note that if a factor of $P(T)$ has a root in $\Delta^{\s,\delta}(a)$, then $P(T)$ has a root in $\Delta^{\s,\delta}(a)$, as proved in \cite[Coroallry 6.3]{MR2025912} ).  To prove the reverse direction, let $S$ be the set of all skew polynomials $P(T)$ with (right) roots in the conjugacy class of $a$. It is enough to show that $S$ is a left Ore set. We need to show that the set  $SQ(T)\cap K[T;\s,\delta]P(T)$ is nonempty for every $P(T)\in S$ and $Q(T)\in K[T;\s,\delta]$. Without loss of generality, we may assume that $P(T)$ is irreducible. If $Q(T)\in K[T;\s,\delta]P(T)$, the proof is trivial. So, assume that $Q(T)\notin K[T;\s,\delta]P(T)$. Let $L(T)$ be the least left common multiple of $P(T)$ and $Q(T)$. Then, $L(T)=P_1(T)Q(T)=Q_1(T)P(T)$ for some $P_1(T), Q_1(T)\in K[T;\s,\delta]$. Since  $K[T;\s,\delta]$ is a UFD, we see that $P_1(T)$ and $P(T)$ must be similar, i.e., the $K[T;\s,\delta]$-modules $K[T;\s,\delta]/K[T;\s,\delta]P_1(T)$ and  $K[T;\s,\delta]/K[T;\s,\delta]P(T)$ are isomorphic. It follows from Lemma 6.4 in \cite{WedderburnII} that every right root of $P_1(T)$ is conjugate to a right root of $P(T)$. Therefore, $P_1(T)$ has no right root in  $\Delta^{\s,\delta}(a)$, and consequently $P_1(T)\in S$. This completes the proof. 
	\end{proof}
We conclude the paper with the following example as an application of the above theorem. 		
	\begin{example}
		Consider the ring $\mathbb{H}[T]$ of polynomials over the skew field of quaternions in a central indeterminate $T$. For any $q_0\in \mathbb{H}$, $Def(q_0)$ is a ring consisting of all $P(T)^{-1}Q(T)\in \mathbb{H}(T)$ such that $P(q)\neq 0$ when $q$ is conjugate to $q_0$.   
	\end{example} 
\end{subsection}
  



\end{section} 

\bibliographystyle{plain}
\bibliography{SkewConvexFunctionRings}

\begin{thebibliography}{10}

\bibitem{aryapoor2022skew}
Masood Aryapoor.
\newblock Skew analysis over quaternions. i.
\newblock {\em arXiv preprint arXiv:2211.07006}, 2022.

\bibitem{MR4586427}
Nabil Bennenni and Andr\'{e} Leroy.
\newblock Evaluation of iterated {O}re polynomials and skew {R}eed-{M}uller
  codes.
\newblock In {\em Algebra and {C}oding {T}heory}, volume 785 of {\em Contemp.
  Math.}, pages 23--34. Amer. Math. Soc., [Providence], RI, 2023.

\bibitem{Freeidealrings}
P.~M. Cohn.
\newblock {\em Free ideal rings and localization in general rings}, volume~3 of
  {\em New Mathematical Monographs}.
\newblock Cambridge University Press, Cambridge, 2006.

\bibitem{ACauchykernel}
Fabrizio Colombo, Graziano Gentili, and Irene Sabadini.
\newblock A {C}auchy kernel for slice regular functions.
\newblock {\em Ann. Global Anal. Geom.}, 37(4):361--378, 2010.

\bibitem{MR58591}
E.~R. Kolchin.
\newblock Galois theory of differential fields.
\newblock {\em Amer. J. Math.}, 75:753--824, 1953.

\bibitem{LamLeory1988}
T.~Y. Lam and A.~Leroy.
\newblock Vandermonde and {W}ronskian matrices over division rings.
\newblock {\em J. Algebra}, 119(2):308--336, 1988.

\bibitem{WedderburnII}
T.~Y. Lam, A.~Leroy, and A.~Ozturk.
\newblock Wedderburn polynomials over division rings. {II}.
\newblock In {\em Noncommutative rings, group rings, diagram algebras and their
  applications}, volume 456 of {\em Contemp. Math.}, pages 73--98. Amer. Math.
  Soc., Providence, RI, 2008.

\bibitem{MR2025912}
T.~Y. Lam and Andr\'{e} Leroy.
\newblock Wedderburn polynomials over division rings. {I}.
\newblock {\em J. Pure Appl. Algebra}, 186(1):43--76, 2004.

\bibitem{lam1994hilbert}
TY~Lam and A~Leroy.
\newblock Hilbert 90 theorems over division rings.
\newblock {\em Transactions of the American Mathematical Society},
  345(2):595--622, 1994.

\bibitem{lamleroy1988algebraic}
TY~Lam and Andr{\'e} Leroy.
\newblock Algebraic conjugacy classes and skew polynomial rings.
\newblock In {\em Perspectives in ring theory}, pages 153--203. Springer, 1988.

\bibitem{leroy2012noncommutative}
Andr{\'e} Leroy.
\newblock Noncommutative polynomial maps.
\newblock {\em Journal of Algebra and its Applications}, 11(04):1250076, 2012.

\bibitem{MR721171}
G\"{u}nter Pilz.
\newblock {\em Near-rings}, volume~23 of {\em North-Holland Mathematics
  Studies}.
\newblock North-Holland Publishing Co., Amsterdam, second edition, 1983.
\newblock The theory and its applications.

\end{thebibliography}

 \end{document}